\documentclass[11pt,reqno]{amsart}
\usepackage[T1]{fontenc}
\usepackage[utf8]{inputenc}
\usepackage{lmodern}
\usepackage{amsmath,amssymb,amsthm}
\usepackage{microtype}
\usepackage[margin=1.05in]{geometry}
\usepackage[hidelinks]{hyperref}
\usepackage[normalem]{ulem}
\usepackage{mathtools}
\usepackage{xcolor}

\hypersetup{pdftitle={A finite threshold for the double-comet conjecture},
  pdfsubject={Minimum adjacency spectral gaps of trees},
  pdfkeywords={tree, adjacency spectral gap, double comet, Perron vector, Schur complement}}
\newtheorem{theorem}{Theorem}[section]
\newtheorem{lemma}[theorem]{Lemma}
\newtheorem{proposition}[theorem]{Proposition}
\newtheorem{corollary}[theorem]{Corollary}
\newtheorem{conjecture}[theorem]{Conjecture}
\theoremstyle{remark}

\numberwithin{equation}{section}
\allowdisplaybreaks[1]
\makeatletter
\let\original@setauthors\@setauthors
\renewcommand{\@setauthors}{%
  \begingroup
  \let\MakeUppercase\@firstofone
  \def\author@andify{%
    \nxandlist
      {\unskip,\penalty-1\space\ignorespaces}%
      {\unskip,\penalty-2\space\ignorespaces}%
      {\unskip,\penalty-2\space\ignorespaces}%
  }%
  \original@setauthors
  \endgroup
}
\makeatother
\title[A finite threshold for the double-comet conjecture]
{A finite threshold for the double-comet conjecture}
\author[Dongxiu Cai]{Dongxiu Cai\textsuperscript{1}}
\author[Zhenbo Chen]{Zhenbo Chen\textsuperscript{1}}
\author[Xiao-Dong Zhang]{Xiao-Dong Zhang\textsuperscript{1,2}}

\thanks{\textsuperscript{1} School of Mathematical Sciences, MOE-LSC,
SHL-MAC, Shanghai Jiao Tong University, 800 Dongchuan Road,
Shanghai 200240, P.R.~China. Emails:
\texttt{diudiutse@sjtu.edu.cn} (Dongxiu Cai) and
\texttt{czb911@sjtu.edu.cn} (Zhenbo Chen).}
\thanks{\textsuperscript{2} Corresponding author: Xiao-Dong Zhang.
Email: \texttt{xiaodong@sjtu.edu.cn}.}
\date{September 8, 2026}

\begin{document}
\begin{abstract}
For a tree $T$, let $g(T)=\lambda_1(T)-\lambda_2(T)$ be the difference
between its two largest adjacency eigenvalues.
A balanced double comet is obtained by attaching equally many
leaves to the two endpoints of a path.
Jovovi\'c, Koledin and Stani\'c conjectured that such a tree attains
the minimum adjacency spectral gap among trees of any fixed order.
We prove that every minimizing tree of order $n\ge200$ is a
balanced double comet.
We also show that, for any finite irreducible reversible
continuous-time Markov chain, the inverse spectral gap differs
from the effective resistance between two states times the stationary
variance of their hitting probability by at most the inverse
Dirichlet gap for killing at those states.
For Perron chains, we give an exact two-vertex Schur-complement
formula for this resistance--variance quantity.
\end{abstract}

\maketitle

\section{Introduction}
\label{sec:introduction}

Let $G$ be a finite simple connected graph on $n\ge2$ vertices. Write its
adjacency eigenvalues in decreasing order as
\[
 \lambda_1(G)>\lambda_2(G)\ge\cdots\ge\lambda_n(G),
 \qquad g(G)=\lambda_1(G)-\lambda_2(G).
\]
For the family $\mathcal T_n$ of trees of order $n$, put
\[
 g_n=\min_{T\in\mathcal T_n}g(T).
\]
For integers $k\ge0$ and $\ell\ge2$, the \emph{balanced double comet}
$C(k,\ell)$ is obtained by attaching $k$ new leaves to each endpoint of
a path on $\ell$ vertices. Its order is $2k+\ell$.
The endpoint vertices of the original path are called its hubs.
In particular, $C(0,n)$ is the path $P_n$.

Jovovi\'c, Koledin and Stani\'c~\cite{JovovicKoledinStanic2018}
proposed the following conjecture and reported its verification by
computer search for $n\le20$.

\begin{conjecture}[Jovovi\'c--Koledin--Stani\'c]
\label{conj:double-comet}
For every $n\ge2$, the minimum adjacency spectral gap among trees of
order $n$ is attained by a balanced double comet.
\end{conjecture}

The two equal leaf clusters give a double comet two comparable regions
of large Perron mass. Increasing the path between them weakens their
interaction and brings the first two eigenvalues close together.
The difficulty is to turn this description into a comparison with
arbitrary trees: bounds on either eigenvalue separately do not control
their possibly much smaller difference.

Our main result gives an explicit threshold for the conjecture and
identifies every minimizer above that threshold.

\begin{theorem}
\label{thm:main}
For every integer $n\ge200$, the minimum adjacency spectral gap
among trees of order $n$ is attained by a balanced double comet.
Thus Conjecture~\ref{conj:double-comet} holds for all $n\ge200$, and
\[
 g_n=
 \min_{0\le k\le\lfloor(n-2)/2\rfloor}
 g(C(k,n-2k)).
\]
Moreover, every tree $T\in\mathcal T_n$ with $g(T)=g_n$ is a balanced
double comet.
\end{theorem}

Theorem~\ref{thm:main} does not assert uniqueness of the minimizing
integer $k$. It also does not settle the conjecture for every smaller
order. Combined with the verification reported in
\cite{JovovicKoledinStanic2018}, it reduces the remaining orders in
Conjecture~\ref{conj:double-comet} to $21\le n\le199$.
The proof of the theorem itself uses no enumeration of trees or
numerical eigenvalue computations.

Spectral centers provide a useful decomposition for questions about the
first two eigenvalues of a tree; see
Kumar, Mohar, Pragada and Zhan~\cite{KumarMoharPragadaZhan2025}.
The same authors study spectral sums and convex combinations in
\cite{KumarMoharPragadaZhan2026}, where the double-comet gap conjecture
is also stated. Related progress on minimum gaps of trees and connected
graphs was announced by Hou~\cite{Hou2026}; that conference abstract
does not specify a theorem or a threshold. The argument here includes
proofs of the particular center and structural statements it needs.

We briefly describe the proof. For a rooted tree $(H,a)$ of radius
$r>2$, let $D$ be the distance from $a$ to its nearest Perron maximum.
The integer
\[
 E=|V(H)|-\lfloor r^2\rfloor-D
\]
is nonnegative. A small value of $E$ restricts branching away from the
root-to-maximum path. At a spectral center of a minimizing tree, this
gives two such budgets and a total defect $\mathcal E$, which also counts
any unused branches at a center vertex. A lower bound for the gap
retains $\mathcal E$ in its exponential term.

A fixed double-comet comparison first implies $\lambda_2(T)>4$ for
$n\ge200$. An optimized comparison then gives $\mathcal E\le3$.
In this range, resolvent estimates place the squared radii of the two
selected components near a common integer degree. A two-dimensional
Schur complement controls the contribution of the other spectral modes.
Comparison with two adjacent double-comet parameters excludes every
positive defect. The common hub degree and zero defect then force both
components to be brooms, whose joined handles form a balanced double comet.

The paper also records a probabilistic estimate that explains the
two-region comparison. The Perron transform conjugates $A-RI$ to the
generator of a reversible continuous-time Markov chain. For any two
boundary vertices $u,v$, let $h$ be the probability of hitting $v$ before
$u$, let $\mathcal R(u,v)$ be their effective resistance in the Perron
conductance network, and let $\gamma$ be the killed Dirichlet gap.
We prove
\[
 0\le\frac1{g(G)}-
 \mathcal R(u,v)\operatorname{Var}_\pi(h)\le\frac1\gamma.
\]
This estimate and its exact two-vertex Schur representation are proved
in Section~\ref{sec:perron-probability}. Section~\ref{sec:comet-comparator}
derives the double-comet bound
$g(C(k,\ell))\le(k-1)^{-(\ell-2)/2}$ for $k\ge3$.
It also gives a direct antisymmetric Rayleigh argument for that bound,
so the main theorem does not depend on the probabilistic lower estimate.
Sections~\ref{sec:rooted-budgets} and~\ref{sec:isolation} establish the
rooted estimates and the spectral separation needed to start the argument.
Section~\ref{sec:center} proves the center criterion, and
Section~\ref{sec:global} establishes its hypotheses for $n\ge200$.

All logarithms are natural. A graph obtained by deleting vertices is
understood to be induced. Spectral radii of nonempty forests are the
largest eigenvalues of their adjacency matrices; the empty forest has
radius zero. We use the Perron--Frobenius theorem, the Rayleigh
principle, eigenvalue interlacing, and Schur complements in their
standard finite-dimensional forms.

\section{Perron chains and a two-boundary estimate}
\label{sec:perron-probability}

Let $G$ be a connected graph on at least two vertices, let $A$ be its
adjacency matrix, and write
\[
 R=\lambda_1(A),\qquad g=R-\lambda_2(A).
\]
Throughout this section, $x$ is the positive Perron vector normalized by
$\sum_v x_v^2=1$, and $X=\operatorname{diag}(x)$.
The transition matrix
\[
 P_{uv}=\frac{A_{uv}x_v}{R x_u}
\]
is stochastic and reversible with respect to $\pi_v=x_v^2$.
Indeed, $Ax=Rx$ gives the row sums, and
$\pi_uP_{uv}=A_{uv}x_ux_v/R$ gives detailed balance.
This transition matrix is also known as the maximal-entropy random walk;
see \cite{BurdaDudaLuckWaclaw2009}.
We use the continuous-time generator
\[
 Q=X^{-1}AX-RI=R(P-I)
\]
and its positive Laplacian $\mathcal L=-Q$.
The map $f\mapsto Xf$ is an isometry from $L^2(\pi)$ to Euclidean space
and conjugates $\mathcal L$ to $RI-A$. Thus the smallest positive
eigenvalue of $\mathcal L$ is exactly $g$.
The conductances and the associated Dirichlet form are
\begin{equation}
 c_{uv}=x_ux_v\quad(uv\in E(G)),\qquad
 \mathcal E(f)=\sum_{uv\in E(G)}c_{uv}(f_u-f_v)^2
 =\langle f,\mathcal Lf\rangle_\pi.
 \label{eq:perron-energy}
\end{equation}
In particular, the variational characterization is
\begin{equation}
 \frac1g=\sup_{f\not\equiv\mathrm{constant}}
 \frac{\operatorname{Var}_\pi(f)}{\mathcal E(f)},\qquad
 \operatorname{Var}_\pi(f)
 =\sum_v\pi_v\Bigl(f_v-\sum_w\pi_w f_w\Bigr)^2.
 \label{eq:perron-variational}
\end{equation}
For a tree, the discrete-time chain $P$ has period two; its absolute
spectral gap is therefore zero. All gap statements below concern the
continuous-time normalization above, so no aperiodicity assumption is
needed.

Fix two distinct vertices $u,v$, and put
$\Gamma=V(G)\setminus\{u,v\}$. Define
\begin{equation}
 \gamma=R-\rho(A_\Gamma)>0
 \label{eq:dirichlet-gamma}
\end{equation}
when $\Gamma$ is nonempty. When $\Gamma$ is empty, set $\gamma=+\infty$
and interpret $\gamma^{-1}=0$ throughout.
The strict inequality in \eqref{eq:dirichlet-gamma} follows from Perron
irreducibility and the fact that $A_\Gamma$ is a proper principal
submatrix. If $r$ vanishes on $\{u,v\}$, restriction of $Xr$ to $\Gamma$
and the Rayleigh principle give
\begin{equation}
 \|r\|_{L^2(\pi)}^2\le\gamma^{-1}\mathcal E(r).
 \label{eq:dirichlet-poincare}
\end{equation}

Let $h$ be the unique function with $h(u)=0$, $h(v)=1$, and
$\mathcal Lh=0$ on $\Gamma$.
Equivalently, $h(w)=\mathbb P_w(\tau_v<\tau_u)$ for the chain with
generator $Q$, where $\tau_a$ denotes the first hitting time of $a$.
The effective resistance between $u$ and $v$ for the conductances
in \eqref{eq:perron-energy} is
\[
 \mathcal R(u,v)=\frac1{\mathcal E(h)}.
\]
This identity can also be taken as its definition: harmonicity makes
$h$ the minimizer of energy among functions with the prescribed boundary
values. Put
\[
 Z_{u,v}=\mathcal R(u,v)\operatorname{Var}_\pi(h).
\]

\begin{theorem}[Two-boundary inverse-gap estimate]
\label{thm:two-boundary}
With the notation above,
\begin{equation}
 0\le\frac1g-Z_{u,v}\le\frac1\gamma.
 \label{eq:two-boundary-bound}
\end{equation}
In particular,
\[
 \frac1{Z_{u,v}+\gamma^{-1}}\le g\le\frac1{Z_{u,v}}.
\]
The constant in the additive error is independent of the number of
vertices.
\end{theorem}

\begin{proof}
The choice $f=h$ in \eqref{eq:perron-variational} proves the lower bound
on $1/g$. For the other direction, write an arbitrary function uniquely as
\[
 f=f(u)+t h+r,\qquad t=f(v)-f(u),\qquad r(u)=r(v)=0.
\]
The bilinear form associated with $\mathcal E$ satisfies
$\mathcal E(h,r)=\langle r,\mathcal Lh\rangle_\pi=0$.
Consequently
\[
 \mathcal E(f)=t^2/\mathcal R(u,v)+\mathcal E(r).
\]
Centering is an orthogonal projection in $L^2(\pi)$. Its triangle
inequality and \eqref{eq:dirichlet-poincare} imply
\[
 \sqrt{\operatorname{Var}_\pi(f)}
 \le |t|\sqrt{\operatorname{Var}_\pi(h)}
       +\sqrt{\mathcal E(r)/\gamma}.
\]
Applying Cauchy--Schwarz to these two terms gives
\[
 \operatorname{Var}_\pi(f)
 \le\bigl(t^2/\mathcal R(u,v)+\mathcal E(r)\bigr)
       \bigl(Z_{u,v}+\gamma^{-1}\bigr).
\]
Now take the supremum in \eqref{eq:perron-variational}.
If $\Gamma$ is empty, then $r=0$, and the same calculation holds with
$\gamma^{-1}=0$.
\end{proof}

The proof uses only reversibility, the Dirichlet variational principle,
and harmonicity. It therefore applies without change to any finite
irreducible reversible continuous-time chain, with $\gamma$ interpreted
as the smallest eigenvalue of its killed Dirichlet operator. For a general
chain with generator $Q$ and stationary probability $\pi$, the
conductances are $c_{uv}=\pi_uQ_{uv}$ for $u\ne v$.

\subsection{An exact two-vertex Schur representation}

Schur reduction of an electrical Laplacian is commonly called Kron
reduction; see \cite{DorflerBullo2013}. We will also retain the derivative
of the adjacency Schur matrix, which records the mass of harmonic
extensions.

Order the vertices with $u,v$ first and write
\[
 A=\begin{pmatrix}A_{\{u,v\}}&B^{\mathsf T}\\B&A_\Gamma\end{pmatrix}.
\]
For $z>\rho(A_\Gamma)$ define
\begin{equation}
 S(z)=zI_2-A_{\{u,v\}}
       -B^{\mathsf T}(zI-A_\Gamma)^{-1}B.
 \label{eq:two-vertex-schur}
\end{equation}
When $\Gamma$ is empty the last term is absent, and this formula is
defined for every $z$.
Set
\begin{equation}
 a=x_u,\quad d=x_v,\quad b=-S_{12}(R),\qquad
 M=S'(R)=I_2+B^{\mathsf T}(RI-A_\Gamma)^{-2}B.
 \label{eq:two-vertex-data}
\end{equation}
Here $M$ is positive definite with nonnegative entries, and $b>0$.
For the latter assertion, either $u$ and $v$ are adjacent, or a
$u$--$v$ path has its internal vertices in a single component of
$\Gamma$, whose resolvent is strictly positive.

\begin{proposition}[Exact Schur formula]
\label{prop:two-vertex}
The normalization and resistance-variance identities are
\begin{equation}
 (a,d)M(a,d)^{\mathsf T}=1,
 \qquad
 Z_{u,v}=\frac{ad\det M}{b}.
 \label{eq:exact-schur-Z}
\end{equation}
Hence, with $\widehat g=b/(ad\det M)$,
\begin{equation}
 \frac{\widehat g}{1+\widehat g/\gamma}\le g\le\widehat g,
 \qquad
 0\le\widehat g-g\le\frac{\widehat g^2}{\gamma+\widehat g}.
 \label{eq:schur-quadratic-error}
\end{equation}
The error term is zero if $\gamma=+\infty$.
There is also the coordinate-free lower bound
\begin{equation}
 g\ge
 \frac{2b}{\sqrt{M_{11}M_{22}}-M_{12}+2b/\gamma}.
 \label{eq:schur-coordinate-free}
\end{equation}
No comparison between $\lambda_2(A)$ and $\rho(A_\Gamma)$ is assumed.
\end{proposition}

\begin{proof}
Define the harmonic extension matrix
\[
 V=\begin{pmatrix}I_2\\(RI-A_\Gamma)^{-1}B\end{pmatrix}.
\]
Block elimination gives $x=V(a,d)^{\mathsf T}$ and $V^{\mathsf T}V=M$,
which proves the normalization identity.
Let $h_u=1-h$ and $h_v=h$. Their adjacency extensions are
\[
 Xh_u=aV e_1,\qquad Xh_v=dV e_2.
\]
Thus their Gram matrix in $L^2(\pi)$ is
\[
 G_0=\begin{pmatrix}a&0\\0&d\end{pmatrix}
       M\begin{pmatrix}a&0\\0&d\end{pmatrix}.
\]
Since $h_u+h_v=1$ and $\pi$ has total mass one, the sum of the entries
of $G_0$ is one. Expanding the determinant now gives
\[
 \operatorname{Var}_\pi(h)=\det G_0=a^2d^2\det M.
\]
Also $S(R)(a,d)^{\mathsf T}=0$, so
\[
 \begin{pmatrix}a&0\\0&d\end{pmatrix}
 S(R)\begin{pmatrix}a&0\\0&d\end{pmatrix}
 =ad b\begin{pmatrix}1&-1\\-1&1\end{pmatrix}.
\]
This is the Schur complement of the electrical Laplacian
$X(RI-A)X$ onto the boundary. Equivalently, evaluating the energy of
the harmonic extension of $(0,1)$ gives $\mathcal E(h)=ad b$.
It follows that $\mathcal R(u,v)=1/(ad b)$, proving
\eqref{eq:exact-schur-Z}.
Theorem~\ref{thm:two-boundary} then gives
\eqref{eq:schur-quadratic-error} by inversion and subtraction.

Finally put $m=\sqrt{M_{11}M_{22}}$ and $c=M_{12}$.
The normalization identity and the arithmetic--geometric mean inequality
give $1\ge 2ad(m+c)$. Therefore
\[
 Z_{u,v}=ad(m^2-c^2)/b\le (m-c)/(2b).
\]
Substituting this upper bound in the lower bound for $g$ proves
\eqref{eq:schur-coordinate-free}. The denominator is positive because
$M$ is positive definite and hence $m>c\ge0$.
\end{proof}

\begin{proposition}[Antisymmetric harmonic certificate]
\label{prop:antisymmetric-certificate}
Suppose $G$ has an automorphism exchanging $u$ and $v$. Write
$M_{11}=M_{22}=m$ and $M_{12}=c$. Then
\begin{equation}
 Z_{u,v}=\frac{m-c}{2b},\qquad
 g\le\widehat g=\frac{2b}{m-c}.
 \label{eq:antisymmetric-certificate}
\end{equation}
The second-order error in \eqref{eq:schur-quadratic-error} applies to
this value of $\widehat g$.
Moreover, the upper bound in \eqref{eq:antisymmetric-certificate} follows
directly from a Rayleigh quotient, independently of
Theorem~\ref{thm:two-boundary}.
\end{proposition}

\begin{proof}
Symmetry gives $a=d$, and \eqref{eq:exact-schur-Z} gives
$a^2=1/[2(m+c)]$. Substitution yields the formula for $Z_{u,v}$.
For the direct argument, let $w=(1,-1)^{\mathsf T}$ and $y_0=Vw$.
Uniqueness of harmonic extension implies that the automorphism sends
$y_0$ to $-y_0$, whereas it fixes $x$. Hence $y_0\perp x$.
Since $S(R)(1,1)^{\mathsf T}=0$, its diagonal entries are both $b$.
Block elimination consequently gives
\[
 y_0^{\mathsf T}(RI-A)y_0=w^{\mathsf T}S(R)w=4b,
 \qquad \|y_0\|^2=w^{\mathsf T}Mw=2(m-c).
\]
The Rayleigh principle on $x^\perp$ proves the asserted upper bound.
\end{proof}

\subsection{Identical rooted ends joined by a path}

Let $(H,u)$ be a rooted tree. Join two copies of $H$ by a path of
$L\ge1$ edges between their roots, using $L-1$ new internal vertices.
Call the resulting tree $G$ and the root of the second copy $v$.
Let $A_0=A_{H-u}$, let $b_0$ be the incidence vector of the neighbors
of $u$ in $H-u$, and define the root response
\[
 \tau(z)=b_0^{\mathsf T}(zI-A_0)^{-1}b_0.
\]
If $H$ consists of its root alone, set $\tau=0$.
For the path polynomials use
\[
 p_{-1}(z)=0,\qquad p_0(z)=1,\qquad
 p_j(z)=z p_{j-1}(z)-p_{j-2}(z)\quad(j\ge1).
\]
Thus $p_j(z)=\det(zI-A_{P_j})$ for $j\ge0$.
Eliminating the path interior gives its endpoint diagonal response and
off-diagonal coupling as
\[
 t_L(z)=\frac{p_{L-2}(z)}{p_{L-1}(z)},\qquad
 b_L(z)=\frac1{p_{L-1}(z)}.
\]
These formulas include $L=1$, for which $t_1=0$ and $b_1=1$.
They follow either from the recurrence for $p_j$ and Cramer's rule, or
by eliminating the internal path vertices consecutively.
For $L=2$ the two endpoints share a single internal vertex; the same
formula gives both responses equal to $1/z$.
The full two-root Schur matrix is therefore
\[
 S(z)=\begin{pmatrix}
 z-\tau(z)-t_L(z)&-b_L(z)\\
 -b_L(z)&z-\tau(z)-t_L(z)
 \end{pmatrix}.
\]
Put
\[
 h_+(z)=z-\tau(z)-t_L(z)+b_L(z).
\]
This is the Schur eigenvalue in the antisymmetric direction. At the
Perron root $R$ of the joined tree,
$h_+'(R)=m-c>0$. Consequently
\begin{equation}
 \widehat g=\frac{2b_L(R)}{h_+'(R)},\qquad
 \frac{\widehat g}{1+\widehat g/\gamma}\le g\le\widehat g.
 \label{eq:identical-caps-certificate}
\end{equation}
Here $\gamma=R-\rho(G-\{u,v\})$ with the same empty-interior
convention as before. Every expression is evaluated above the spectra
of the eliminated blocks. In particular,
\eqref{eq:identical-caps-certificate} does not require $\lambda_2(G)>2$.

\section{A uniform upper bound for balanced double comets}
\label{sec:comet-comparator}

For integers $k\ge1$ and $\ell\ge2$, let $C(k,\ell)$ be the tree obtained
from a path on $\ell$ vertices by adjoining $k$ leaves to each endpoint.
It has $2k+\ell$ vertices. Its two endpoints will be called hubs.

\begin{theorem}
\label{thm:comet-upper}
For every $k\ge3$ and $\ell\ge2$,
\begin{equation}
 g(C(k,\ell))\le(k-1)^{-(\ell-2)/2}.
 \label{eq:comet-factor-one}
\end{equation}
Equality holds when $\ell=2$. For $\ell\ge3$ one has the stronger bound
\begin{equation}
 g(C(k,\ell))<
 \left(1-\frac{2}{k(k+1)}\right)(k-1)^{-(\ell-2)/2}.
 \label{eq:comet-factor-one-strong}
\end{equation}
\end{theorem}

\begin{proof}
Write $L=\ell-1$. If $L=1$, for nonzero eigenvalues the two hub equations are
$z-k/z=\pm1$. The two positive eigenvalues are
\[
 \frac{\sqrt{4k+1}+1}{2},\qquad
 \frac{\sqrt{4k+1}-1}{2},
\]
and all remaining eigenvalues are their negatives or zero.
Their difference is $1$, proving the equality assertion.

Assume henceforth that $L\ge2$. The root response of either leaf cluster
is $\tau(z)=k/z$, so the symmetric Schur eigenvalue is
\[
 F_L(z)=z-\frac{k}{z}-t_L(z)-b_L(z),\qquad z>2.
\]
All eliminated blocks have spectral radius below $2$.
The Schur derivative is positive definite, hence $F_L'(z)>0$.
Any zero of $F_L$ above $2$ gives a positive eigenvector by harmonic
extension of the hub values $(1,1)$, and is therefore the Perron root.
For $z=q+q^{-1}$ with $0<q<1$, the path recurrence gives
\[
 p_{L-1}(z)=q^{-(L-1)}\frac{1-q^{2L}}{1-q^2},
 \qquad
 t_L(z)+b_L(z)=\frac{q+q^{L-1}}{1+q^L}.
\]
In particular,
\[
 q<t_L(z)+b_L(z)\le\frac2z.
\]
The first inequality follows by subtracting $q$; the second follows from
\[
 \frac2z-t_L(z)-b_L(z)
 =\frac{(1-q^2)(q-q^{L-1})}{(1+q^2)(1+q^L)}\ge0.
\]
Put $t=k-1$, $q_*=t^{-1/2}$, $\Lambda=q_*+q_*^{-1}>2$, and
$z_0=\sqrt{k+2}$. Since $\Lambda-k/\Lambda=q_*$ and
$z_0-k/z_0=2/z_0$, we have $F_L(\Lambda)<0\le F_L(z_0)$.
Thus $F_L$ has a unique zero in $(\Lambda,z_0]$, and
\begin{equation}
 \Lambda<R\le\sqrt{k+2}.
 \label{eq:comet-root-upper}
\end{equation}

Proposition~\ref{prop:antisymmetric-certificate}, or its direct odd
harmonic Rayleigh argument, gives
\[
 g\le\frac{2b_L(R)}{h_+'(R)},\qquad
 h_+(z)=z-\frac{k}{z}-t_L(z)+b_L(z).
\]
In the Schur derivative \eqref{eq:two-vertex-data}, the leaves contribute
$k/R^2$ in each hub direction, and the internal path contributes a
positive semidefinite matrix. Hence
\[
 h_+'(R)\ge1+\frac{k}{R^2}\ge\frac{2(k+1)}{k+2}.
\]
For $z>2$, the polynomial $p_{L-1}(z)$ is positive and strictly
increasing, since its logarithmic derivative is
$\operatorname{tr}(zI-A_{P_{L-1}})^{-1}>0$.
Therefore
\[
 b_L(R)<b_L(\Lambda)
 =\frac{1-t^{-1}}{1-t^{-L}}t^{-(L-1)/2}
 \le\frac{t}{t+1}t^{-(L-1)/2},
\]
where the last inequality uses $L\ge2$. Combining these estimates yields
\[
 g<\frac{(k-1)(k+2)}{k(k+1)}(k-1)^{-(L-1)/2},
\]
which is \eqref{eq:comet-factor-one-strong}.
\end{proof}

\begin{corollary}
\label{cor:small-gap}
Let $g_n$ be the minimum adjacency spectral gap among all trees on
$n$ vertices. Then $g_n\le n^{-3}$ for every integer $n\ge35$.
\end{corollary}

\begin{proof}
Set $k=\lfloor n/4\rfloor$ and $\ell=n-2k$.
For $n\ge35$ these parameters are admissible in
Theorem~\ref{thm:comet-upper}, and
\[
 k-1\ge\frac{n-7}{4}\ge\frac n5,
 \qquad \frac{\ell-2}{2}\ge\frac n4-1\ge\frac n5.
\]
Also $\log(n/5)\ge\tfrac12\log n$ for $n\ge25$.
Consequently
\[
 \frac{\ell-2}{2}\log(k-1)
 \ge\frac n{10}\log n\ge3\log n.
\]
Apply \eqref{eq:comet-factor-one} to $C(k,\ell)$.
\end{proof}

\section{Rooted distance and path budgets}
\label{sec:rooted-budgets}

Throughout, a tree is finite, simple, and connected.  For a graph $H$, write
$A_H$ for its adjacency matrix and $\rho(H)$ for its spectral radius.
All matrix norms below are operator norms on Euclidean space.  We use two
elementary consequences of bipartiteness repeatedly.  If a forest has $e$
edges, then
\begin{equation}
 \rho(H)^2\le \sum_{\theta>0}\theta^2=e,
 \label{eq:forest-edge-bound}
\end{equation}
where the sum includes multiplicities.  Indeed, the spectrum is symmetric
about zero and $\operatorname{tr}(A_H^2)=2e$.  In particular, a tree of order
$h$ satisfies $\rho(H)^2\le h-1$.  Also, if $p$ is the unit positive Perron
vector of a nontrivial tree, its squared masses on the two bipartition
classes are both $1/2$.  To see this, write its Perron equations as
$r p_X=Cp_Y$ and $r p_Y=C^{\mathsf T}p_X$, and take the scalar products with
$p_X$ and $p_Y$.  Consequently
\begin{equation}
 p_v^2\le\frac12 \qquad(v\in V(H)).
 \label{eq:bipartite-coordinate-bound}
\end{equation}

\subsection{The distance from a root to a Perron maximum}

For a rooted tree $(H,a)$, let $r=\rho(H)>2$, let $p$ be its unit positive
Perron vector, and choose a vertex $w$ with maximum Perron coordinate that
is nearest to $a$.  Set
\begin{equation}
 h=|V(H)|,\qquad D=\operatorname{dist}(a,w),\qquad
 N=\lfloor r^2\rfloor,\qquad E=h-N-D.
 \label{eq:rooted-budget-notation}
\end{equation}
Write $Q$ for the $a$--$w$ path, including its endpoints.  Distances are
measured in edges.  A \emph{broom relative to $w$} is a tree obtained from
$Q$ by attaching leaves only at $w$.

\begin{lemma}[Rooted distance budget]
\label{lem:rooted-budget}
With the notation in \eqref{eq:rooted-budget-notation},
\[
 D\le h-\lfloor r^2\rfloor,
 \qquad\text{equivalently } E\ge0.
\]
\end{lemma}

\begin{proof}
The cases $D=0,1$ follow from $r^2\le h-1$.  Suppose $D\ge2$, and put
$d=\deg_H(w)$ and $s=|\{v:\operatorname{dist}(w,v)=2\}|$.  The squared
Perron equation at $w$ reads
\begin{equation}
 r^2=d+\sum_{\operatorname{dist}(w,v)=2}\frac{p_v}{p_w}.
 \label{eq:squared-perron-row}
\end{equation}
All ratios are at most one.  The vertex of $Q$ at distance two from $w$
has a strictly smaller coordinate than $w$, by the nearest-maximum choice.
Thus $r^2<d+s$, so the radius-two ball $K$ about $w$ has at least $N+2$
vertices.  Its intersection with $Q$ consists of exactly three vertices.
It follows that
\[
 h\ge |V(K)\cup V(Q)|=|V(K)|+D-2\ge N+D.
\]
\end{proof}

\begin{lemma}[Distance-two deficit]
\label{lem:distance-two}
In the same notation, let $s$ be the number of vertices at distance two
from $w$. Then
\begin{equation}
 s\le E+1+\frac{E+1}{1-1/r}
   =(E+1)\frac{2r-1}{r-1}.
 \label{eq:distance-two-deficit}
\end{equation}
In particular, if $r>4$ and $E\le3$, then $s\le9$.
\end{lemma}

\begin{proof}
Write $d=\deg_H(w)$ and let $K$ be the radius-two ball about $w$.
Counting $K\cup Q$ and using \eqref{eq:squared-perron-row} gives
\[
 d+s\le N+E+1,\qquad 0\le d+s-r^2\le E+1.
\]
Let $q$ count the vertices outside both $K$ and $Q$. If $D=0$ or $D=1$,
then $q=N+E+D-d-s-1\le E$, since $d+s\ge r^2\ge N$.
If $D\ge2$, the strict inequality in the proof of
Lemma~\ref{lem:rooted-budget} gives $d+s\ge N+1$, and hence
$q=N+E+1-d-s\le E$.

Every nonleaf vertex at distance two has a neighbor at distance three.
Choose one such neighbor for each of these vertices. The chosen neighbors
are distinct because $H$ is a tree, and at most one belongs to $Q$.
All the others are counted by $q$, so there are at most $E+1$ nonleaf
vertices at distance two. For a leaf $v$ at distance two, the Perron
equation gives $p_v\le p_w/r$. Thus each such leaf contributes at least
$1-1/r$ to
\[
 d+s-r^2
 =\sum_{\operatorname{dist}(w,v)=2}\left(1-\frac{p_v}{p_w}\right)
 \le E+1.
\]
There are at most $(E+1)/(1-1/r)$ such leaves, proving
\eqref{eq:distance-two-deficit}. If $r>4$ and $E\le3$, this bound gives
$s<28/3$, so integrality yields $s\le9$.
\end{proof}

\subsection{A path budget for the spectral radius}

\begin{lemma}[Path budget]
\label{lem:path-budget}
Let $H$ be a tree, let $Q$ be any nonempty path in $H$ (a single vertex is
allowed), and put $m=|V(H)\setminus V(Q)|$.  Then
\begin{equation}
 \rho(H)^2<\Phi(m),\qquad
 \Phi(m)=2+\sqrt{m^2+4}.
 \label{eq:sharp-path-budget}
\end{equation}
More generally, if $H$ is a forest containing a specified nonempty path
$Q$ and $e$ is the number of its edges outside $Q$, then
$\|A_H\|^2<\Phi(e)$.  If no path is specified, a forest with $e$ edges
satisfies $\|A_H\|^2\le e<\Phi(e)$.
\end{lemma}

\begin{proof}
Write $\lambda=\rho(H)$.  If $m=0$, then $H=Q$ and $\lambda<2$.
For $m\ge1$, the conclusion is immediate if $\lambda\le2$ or
$\lambda^2\le m$.  Thus assume $\lambda>2$ and $\lambda^2>m$.
Each component $F_j$ of $H-V(Q)$ is attached to $Q$ by exactly one edge.
Let $t_j=|V(F_j)|$, and let $b_j$ be the endpoint of its attachment edge
in $F_j$.  By \eqref{eq:forest-edge-bound}, $\rho(F_j)^2\le t_j-1$.
Bipartiteness, or the expansion into closed walks at $b_j$, gives
\begin{align*}
 R_j(\lambda)
 &:=\left\langle e_{b_j},(\lambda I-A_{F_j})^{-1}e_{b_j}\right\rangle\\
 &=\left\langle e_{b_j},
       \lambda(\lambda^2I-A_{F_j}^2)^{-1}e_{b_j}\right\rangle
 \le\frac{\lambda}{\lambda^2-t_j+1}
 \le\frac{t_j}{\lambda}.
\end{align*}
The last inequality is equivalent to
$(t_j-1)(\lambda^2-t_j)\ge0$, which holds because $t_j\le m<\lambda^2$.

Eliminate all the components $F_j$ from the Perron equations.  If $x>0$
is the restriction of the Perron vector to $Q$, the result is
\[
 (\lambda I-A_Q)x=Vx,
\]
where $V$ is a nonnegative diagonal matrix and
$\operatorname{tr}V=\sum_j R_j(\lambda)\le m/\lambda$.
Set $G=(\lambda I-A_Q)^{-1}$.  The positive semidefinite matrix
$G^{1/2}VG^{1/2}$ has eigenvalue one, because $GVx=x$ and these two
matrices are similar.  Therefore
\[
 1\le\operatorname{tr}(GV).
\]
For every vertex $v$ of a finite path,
\begin{equation}
 G_{vv}<\frac1{\sqrt{\lambda^2-4}}.
 \label{eq:path-green-diagonal}
\end{equation}
Indeed, every closed walk at $v$ embeds into the doubly infinite path,
and some sufficiently long even walks are missing.  The nonnegative
resolvent series and the central-binomial generating function give
\[
 G_{vv}<\frac1\lambda\sum_{j\ge0}
       \binom{2j}{j}\lambda^{-2j}
       =\frac1{\sqrt{\lambda^2-4}}.
\]
This argument also covers a one-vertex path.  Since $V$ has a positive
diagonal entry, we obtain
\[
 1\le\operatorname{tr}(GV)
 <\frac{m}{\lambda\sqrt{\lambda^2-4}}.
\]
Squaring and solving this inequality for $\lambda^2>4$ proves
\eqref{eq:sharp-path-budget}.

For the forest assertion, the component containing $Q$ has as many
vertices outside $Q$ as it has edges outside $Q$.  Apply the proved bound
to that component.  Every other component has squared spectral radius
at most its number of edges, by \eqref{eq:forest-edge-bound}.  Each of
these edge counts is at most $e$, and $\Phi$ is increasing.  Taking the
largest component norm proves the assertion.  The last statement is
\eqref{eq:forest-edge-bound} itself.
\end{proof}

\subsection{One-vertex Schur estimates}

\begin{lemma}[Rooted Schur estimates]
\label{lem:rooted-schur}
Let $H$ be a nontrivial tree with spectral radius $r$, let $p$ be its unit
positive Perron vector, and let $w$ be any vertex.  Put
\[
 d=\deg_H(w),\qquad s=|\{v:\operatorname{dist}(w,v)=2\}|,
 \qquad B=A_{H-w}.
\]
If $r^2>M\ge\|B\|^2$, then
\begin{equation}
 0\le r^2-d\le\frac{s}{r^2-M},\qquad
 p_w^2\ge\left(2+\frac{2s}{(r^2-M)^2}\right)^{-1}.
 \label{eq:rooted-schur-bounds}
\end{equation}
\end{lemma}

\begin{proof}
Let $b$ be the indicator vector of the neighbors of $w$ in $H-w$.
Since $H$ has no cycles, $Bb$ is the indicator vector of the vertices
at distance two from $w$.  In particular, $b^{\mathsf T}b=d$ and
$\|Bb\|^2=s$.  Write $c_j=b^{\mathsf T}B^j b$.  The neighbors of $w$
belong to the same bipartition class, so $c_{2j+1}=0$; for $j\ge1$,
\[
 0\le c_{2j}=\|B^j b\|^2\le s M^{j-1}.
\]
Strict Perron monotonicity gives $\|B\|<r$.  The scalar Schur function
\[
 F(z)=z-b^{\mathsf T}(zI-B)^{-1}b
\]
satisfies $F(r)=0$.  Expanding this equation gives
\[
 r^2=d+\sum_{j\ge1}\frac{c_{2j}}{r^{2j}},
\]
and the geometric-series bound proves the first inequality in
\eqref{eq:rooted-schur-bounds}.

The eigenvector on $H-w$, normalized by $p_w$, is $(rI-B)^{-1}b$.
Consequently $p_w^{-2}=F'(r)$.  Eliminating $d$ from the derivative by
the preceding equation yields the exact identity
\[
 F'(r)=2+\sum_{j\ge1}\frac{2j\,c_{2j}}{r^{2j+2}}
 \le2+\frac{2s}{r^4}
       \sum_{j\ge1}j\left(\frac{M}{r^2}\right)^{j-1}
 =2+\frac{2s}{(r^2-M)^2}.
\]
Taking reciprocals proves the second inequality.
\end{proof}

\section{Isolation above spectral radius two}
\label{sec:isolation}

We give the needed spectral isolation argument directly.  In particular,
no classification theorem or list of generated trees is used below.
For positive integers $a,b,c$, denote by $Y(a,b,c)$ the tree with one
vertex of degree three and three pendant paths of $a,b,c$ edges.

\begin{lemma}[Spectral isolation]
\label{lem:isolation}
Every tree $H$ with $\rho(H)>2$ satisfies
\begin{equation}
 \rho(H)>2+\frac1{162}.
 \label{eq:radius-isolation}
\end{equation}
Moreover, if $T$ minimizes the adjacency spectral gap among trees of
its order $n\ge35$, then
\begin{equation}
 \lambda_2(T)>2,\qquad \rho(T)>2+\frac1{162}.
 \label{eq:minimizer-isolation}
\end{equation}
\end{lemma}

\begin{proof}
We first prove \eqref{eq:radius-isolation}.  The following leaf-extension
estimate is useful.  Suppose $H_0$ has spectral radius two and positive
Perron vector $x$, and put $S=\sum_v x_v^2$.  Add a leaf at a vertex $a$
and extend $x$ to that leaf with coordinate $x_a/2$.  Its Rayleigh quotient
gives
\begin{equation}
 \rho(H_0+\text{a leaf at }a)-2
 \ge\frac{x_a^2}{2S+x_a^2/2}.
 \label{eq:leaf-extension-isolation}
\end{equation}
Indeed, the extended vector has squared norm $S+x_a^2/4$ and adjacency
quadratic form $2S+x_a^2$.

If $\Delta(H)\ge5$, the tree contains $K_{1,5}$.  Apply
\eqref{eq:leaf-extension-isolation} to $K_{1,4}$ with central coordinate
one and leaf coordinates $1/2$; here $S=2$.  Adding a leaf at the center
gives $\rho(H)\ge2+2/9$.  If $\Delta(H)=4$, a copy of $K_{1,4}$ must
be properly extended, because $\rho(H)>2$.  A shortest extension adds
a leaf at one of its leaves, since its center already has degree four.
The same estimate gives $\rho(H)\ge2+2/33$.

We may now suppose $\Delta(H)\le3$.  If there is just one degree-three
vertex, then $H=Y(a,b,c)$ with $a\le b\le c$.  The endpoint resolvent
of a $j$-vertex path at two is $j/(j+1)$, as follows from the path
determinant recurrence.  Eliminating the three arms therefore shows
\begin{equation}
 \rho(Y(a,b,c))>2
 \quad\Longleftrightarrow\quad
 \frac1{a+1}+\frac1{b+1}+\frac1{c+1}<1.
 \label{eq:spider-radius-test}
\end{equation}
This inequality forces $H$ to contain one of
$Y(1,2,6)$, $Y(1,3,4)$, and $Y(2,2,3)$.
Indeed, $a\ge2$ forces $c\ge3$ unless $a=b=c=2$;
$a=1,b=1$ is impossible; $a=1,b=2$ forces $c\ge6$; and
$a=1,b\ge3$ forces $c\ge4$ unless $b=c=3$.
The last two listed trees extend the radius-two trees $Y(1,3,3)$ and
$Y(2,2,2)$ by one leaf.  Normalize each junction coordinate to one.
The squared norms and coordinates at the extended endpoints are
respectively $(S,x_a)=(3,1/4)$ and $(8/3,1/3)$.
Equation~\eqref{eq:leaf-extension-isolation} gives increases at least
$2/193$ and $2/97$, respectively.

Suppose there are exactly two degree-three vertices, at distance $L$.
The four paths leaving their connecting path are nonempty.  If all four
have one edge, the radius is two: assign one to each connecting-path
vertex and $1/2$ to each of the four leaves.  Otherwise at least one
pendant path has two edges.  If $L\ge5$, the tree contains $Y(1,2,6)$:
use that two-edge arm, the other arm at the same junction, and a six-edge
path towards and, if necessary, beyond the second junction.
If $L\le4$, it contains the radius-two double fork just described, with
one of its leaves extended by an edge.  The old squared norm is
$S=L+2\le6$ and the extended coordinate is $1/2$; thus its radius is
at least $2+2/97$.

If there are at least three degree-three vertices, choose three of them,
$u,v,w$, occurring in this order on a path, with no further degree-three
vertices in its two open segments.  For justification, form the auxiliary
tree whose vertices are the degree-three vertices of $H$, joining two
when their connecting path contains no other degree-three vertex.
It is a tree on at least three vertices: it is obtained from the minimal
subtree joining these vertices by suppressing only the vertices not in
this specified set.  Choose a two-edge path in this auxiliary tree.
Put $L_1=\operatorname{dist}(u,v)$ and $L_2=\operatorname{dist}(v,w)$.
If either length is at least five, the tree contains $Y(1,2,6)$ centered
at $v$: take a six-edge arm towards the more distant outer junction, a
two-edge arm towards the other, and one edge in the third direction.
If both lengths are at most four, retain the $u$--$w$ path, two off-path
neighbors at each endpoint, and one off-path neighbor at $v$.
This is a radius-two double fork with an extra leaf at an internal path
vertex.  Its old squared norm is $S=L_1+L_2+2\le10$ and $x_v=1$, so
\eqref{eq:leaf-extension-isolation} gives a radius at least $2+2/41$.
Trees without a branching vertex are paths and have radius below two.

It remains to check $Y(1,2,6)$.  Let $p_j(x)$ be the characteristic
polynomial of the $j$-vertex path, with $p_0=1$, $p_1=x$, and
$p_j=xp_{j-1}-p_{j-2}$.  Above two, eliminating the three arms gives
the strictly increasing Schur function
\[
 F(x)=x-\frac1x-\frac{x}{x^2-1}-\frac{p_5(x)}{p_6(x)}.
\]
Its denominator $x(x^2-1)p_6(x)$ is positive, and the numerator is
\[
 P(x)=x^{10}-9x^8+27x^6-31x^4+12x^2-1.
\]
Writing $s=x^2-4$, direct expansion gives
\[
 P(x)=-1+36s+69s^2+43s^3+11s^4+s^5.
\]
At $x=2+1/162$, we have $s=649/26244<1/40$, whence
\[
 P(x)<-1+\frac{36}{40}+\frac{71}{40^2}
       =-\frac{89}{1600}<0.
\]
Here $43/40+11/40^2+1/40^3<2$ bounds the last three terms by
$2/40^2$.  Hence the Perron root of $Y(1,2,6)$ exceeds $2+1/162$.
Every other structural case gave an increase at least $2/193>1/162$.
Monotonicity under taking subtrees proves \eqref{eq:radius-isolation}.

We finish by proving \eqref{eq:minimizer-isolation}.  A brief classification
of trees with radius at most two suffices for the orders under consideration.
A degree at least five is impossible.  A degree-four vertex forces the
whole tree to be $K_{1,4}$ by strict Perron monotonicity.
If there are two degree-three vertices, their connecting path and two
off-path neighbors at each end induce a radius-two double fork, so the
whole tree must be that double fork.  With exactly one degree-three
vertex, the reverse non-strict inequality in
\eqref{eq:spider-radius-test} leaves only
\[
 Y(1,1,c)\ (c\ge1),\quad
 Y(1,2,c)\ (2\le c\le5),\quad
 Y(1,3,3),\quad Y(2,2,2).
\]
For verification, $a\ge2$ forces $(a,b,c)=(2,2,2)$, while $a=1$
leaves $b=1$, or $b=2,c\le5$, or $b=c=3$.
The exceptional trees in this list have at most nine vertices.
Consequently, at every order $n\ge10$, the only possibilities are
$P_n$, $Y(1,1,n-3)$, and $C(2,n-4)$.

Their gaps are, respectively,
\begin{align*}
 4\sin(3x)\sin x,\quad &x=\frac{\pi}{2(n+1)},\\
 4\sin(2y)\sin y,\quad &y=\frac{\pi}{2(n-1)},\\
 4\sin^2 z,\quad &z=\frac{\pi}{2(n-3)}.
\end{align*}
Here is one direct verification of the latter two formulas.
Combining the two equal leaf coordinates of $Y(1,1,n-3)$ gives a weighted
path with characteristic polynomial $p_{n-1}-p_{n-3}$, together with an
antisymmetric zero eigenvalue.  At $\lambda=2\cos\theta$, the polynomial
equals $2\cos((n-1)\theta)$, so its two largest eigenvalues are
$2\cos y$ and $2\cos(3y)$.  Combining both leaf pairs of $C(2,n-4)$
gives a weighted path with characteristic polynomial
$(\lambda^2-4)p_{n-4}(\lambda)$ and two additional zero eigenvalues.
Its largest two eigenvalues are $2$ and $2\cos(\pi/(n-3))$.
The displayed determinant identities follow by the same path recurrence
$p_j=\lambda p_{j-1}-p_{j-2}$, with the end weights $\sqrt2$.

Using $\sin t\ge2t/\pi$ on $[0,\pi/2]$, these three gaps are at least
\[
 \frac{12}{(n+1)^2},\qquad
 \frac8{(n-1)^2},\qquad
 \frac4{(n-3)^2},
\]
respectively.  Each is greater than $n^{-2}$.
In contrast, Corollary~\ref{cor:small-gap}, which uses only an explicit
balanced-comet competitor, gives $g(T)\le n^{-3}$ for a minimizing tree
of order $n\ge35$.  Thus such a minimizer has $\rho(T)>2$ and hence
$\rho(T)>2+1/162$ by the first part of the lemma.  Finally
\[
 \lambda_2(T)=\rho(T)-g(T)
 >2+\frac1{162}-n^{-3}>2,
\]
as required.
\end{proof}

\section{Spectral centers and a finite structural criterion}
\label{sec:center}

Throughout this section, $T$ is a tree with largest two adjacency eigenvalues
$R=\lambda_1(T)$ and $\beta=\lambda_2(T)>2$, and $g=R-\beta$.
We first give the center decomposition and the estimates on which the
structural argument rests. These statements do not require $\beta$ to be
simple or an associated eigenvector to be nonzero at every vertex.

\begin{lemma}[Spectral-center dichotomy]
\label{lem:center-dichotomy}
At least one of the following alternatives holds.
\begin{enumerate}
\item There is an edge $a_1a_2$ such that the two components $H_1,H_2$ of
$T-a_1a_2$, rooted at $a_1,a_2$, have spectral radii $r_1,r_2>\beta$.
\item There is a vertex $v$ such that every component of $T-v$ has spectral
radius at most $\beta$, and at least two components have spectral radius
exactly $\beta$.
\end{enumerate}
In the second alternative, select any two such components and root them at
their neighbors of $v$. We call these the edge-center and vertex-center
alternatives, respectively.
\end{lemma}

\begin{proof}
For every vertex $v$, at most one component of $T-v$ has spectral radius
larger than $\beta$. Otherwise the principal submatrix on $T-v$ would have
two eigenvalues larger than $\lambda_2(T)$, contrary to interlacing.
If every vertex has exactly one such component, orient each vertex toward
the neighbor belonging to this component. Every vertex has one outgoing edge. A directed
cycle exists, and the underlying tree forces a two-cycle. Its edge gives
the first alternative.

Otherwise choose $v$ for which every component has radius at most $\beta$,
and let $x$ be a nonzero $\beta$-eigenvector of $T$. If every component $H$
had radius less than $\beta$, its eigenvector equation would give
\[
 x|_H=x_v(\beta I-A_H)^{-1}e_a,
\]
where $a$ is its attachment vertex. The inverse has strictly positive entries:
this follows from its convergent nonnegative walk expansion and the
connectedness of $H$. Thus $x_v=0$ would imply $x=0$, whereas $x_v\ne0$
would make all entries of $x$ have the same sign. Both are impossible, the
second by orthogonality to the positive Perron vector of $T$.

Hence some branch $H$ has radius $\beta$. Multiplication of its eigenvector
equation by a positive Perron vector of $H$ gives $x_v=0$. All branches of
smaller radius then carry zero entries. On every branch of radius $\beta$,
the restriction of $x$ is a multiple of its Perron vector. The equation at
$v$ says that the sum of these multiples, weighted by their positive
attachment coordinates, is zero. A nonzero $x$ therefore uses at least two
such branches. This proves the second alternative, including the cases of
zero coordinates and multiple $\beta$.
\end{proof}

Fix a decomposition supplied by Lemma~\ref{lem:center-dichotomy}. Write
$h_i=|H_i|$, let $p_i$ be its positive unit Perron vector, and set
$\alpha_i=p_i(a_i)$. In both cases $\beta\le r_i<R$.
Choose a vertex $w_i$ maximizing $p_i$, with the smallest distance from $a_i$,
and define
\begin{equation}
\label{eq:center-defect}
 D_i=\operatorname{dist}(a_i,w_i),\qquad
 N_i=\lfloor r_i^2\rfloor,\qquad
 E_i=h_i-N_i-D_i.
\end{equation}
The rooted-distance bound in Lemma~\ref{lem:rooted-budget} gives $E_i\ge0$.
Let $\chi=0$ in the edge-center case and $\chi=1$ in the vertex-center case,
and put
\begin{equation}
\label{eq:center-total-defect}
 z=n-\chi-h_1-h_2,\qquad
 \mathcal E=E_1+E_2+z.
\end{equation}
Here $z=0$ for an edge center, while $z$ counts the vertices in the unselected
branches for a vertex center. In particular,
\begin{equation}
\label{eq:center-distance-identity}
 D_1+D_2+\chi=n-N_1-N_2-\mathcal E.
\end{equation}

\begin{lemma}[Center compression bounds]
\label{lem:center-compression}
The two alternatives satisfy, respectively,
\begin{equation}
\label{eq:center-compression}
 g>\alpha_1\alpha_2,
 \qquad
 g\ge\frac{\alpha_1^2+\alpha_2^2}{R}
       \ge\frac{2\alpha_1\alpha_2}{R}.
\end{equation}
The vertex-center bounds also hold when some branches are unselected.
\end{lemma}

\begin{proof}
In the edge case, compression to the two Perron vectors gives
\[
 \begin{pmatrix}r_1&\alpha_1\alpha_2\\
                 \alpha_1\alpha_2&r_2\end{pmatrix}.
\]
Its largest eigenvalue is greater than $\beta+\alpha_1\alpha_2$.
Rayleigh--Ritz proves the first inequality. In the vertex case, compression
to the same two vectors and $e_v$ gives
\[
 \begin{pmatrix}
 \beta&0&\alpha_1\\0&\beta&\alpha_2\\
 \alpha_1&\alpha_2&0
 \end{pmatrix}.
\]
Its largest eigenvalue $\mu$ satisfies
$\mu(\mu-\beta)=\alpha_1^2+\alpha_2^2$.
Since $R\ge\mu>\beta>0$, we have
$R(R-\beta)\ge\alpha_1^2+\alpha_2^2$. The last inequality is
arithmetic--geometric mean.
\end{proof}

\begin{lemma}[Universal Perron amplitude]
\label{lem:center-amplitude}
Let $H$ be a tree of order $h$ with spectral radius $r>1$, and let $p$ be
its positive unit Perron vector. Then
\begin{equation}
\label{eq:center-amplitude}
 \max_v p_v^2\ge\frac{r(r-1)}{2(h-r-1)}.
\end{equation}
\end{lemma}

\begin{proof}
Normalize the Perron vector as $u_w=\max_v u_v=1$, and put
$S_1=\sum_v u_v$, $S_2=\sum_v u_v^2$. Summing its eigenvector equations gives
\[
 (r-1)S_1=\sum_v(\deg(v)-1)u_v\le h-2.
\]
Orient the tree away from $w$. Every parent coordinate is at most one, so
\[
 rS_2=2\sum_{ab\in E(H)}u_au_b
 \le2\sum_{v\ne w}u_v=2(S_1-1).
\]
Combining the inequalities yields
$S_2\le2(h-r-1)/(r(r-1))$. Since $\max_v p_v^2=1/S_2$, this proves the
claim. The denominator is positive; for example,
$r^2\le h-1$ and $r>1$ imply $h-r-1>0$. The elementary tree bound used here
follows from bipartite spectral symmetry and
$\sum_{\lambda_j(H)>0}\lambda_j(H)^2=h-1$.
\end{proof}

\begin{lemma}[Transfer along a rooted path]
\label{lem:root-transfer}
Let $H$ have spectral radius $r>2$ and positive Perron vector $p$.
For vertices $a,w$ at distance $D$, and any $R\ge r$, put
$q=(R-\sqrt{R^2-4})/2$. Then
\begin{equation}
\label{eq:center-transfer}
 p_a\ge\frac{p_w}{p_D(R)}
       \ge p_w(1-q^2)q^D,
\end{equation}
where $p_j(z)=\det(zI-A_{P_j})$ are the path polynomials introduced
in Section~\ref{sec:perron-probability}, with $p_0(z)=1$.
\end{lemma}

\begin{proof}
The case $D=0$ is immediate, so assume $D\ge1$.
Write the coordinates on the $a$--$w$ path as $x_0,\ldots,x_D$.
The eigenvector equations give
$x_1=Rx_0-b_0$ and
$x_{j+1}=Rx_j-x_{j-1}-b_j$, where every $b_j\ge0$; these terms include
both the off-path neighbors and $(R-r)x_j$.
Solving the recurrence yields
\[
 x_D=p_D(R)x_0-\sum_{j=0}^{D-1}p_{D-1-j}(R)b_j
 \le p_D(R)x_0,
\]
since $p_j(R)>0$ for $R>2$. Finally, the path-polynomial formula gives
\[
 p_D(R)=\frac{q^{-D}(1-q^{2D+2})}{1-q^2}
 \le\frac{q^{-D}}{1-q^2}.
\]
\end{proof}

\begin{lemma}[A center bound with an integer defect]
\label{lem:center-bound}
Put $q=(R-\sqrt{R^2-4})/2$, $y=q^{-2}$, and $\eta=2Rg$. Then
\begin{equation}
\label{eq:center-universal}
 g\ge
 \frac{\beta(\beta-1)}{n-2\beta-2}
 (1-1/y)^2y^{-(n-N_1-N_2-\mathcal E)/2}.
\end{equation}
Moreover,
\begin{equation}
\label{eq:center-integer-budget}
 0\le R^2-r_i^2\le\eta,
 \qquad N_1+N_2\ge2R^2-2\eta-2.
\end{equation}
\end{lemma}

\begin{proof}
Lemma~\ref{lem:center-amplitude} and arithmetic--geometric mean give
\[
 p_1(w_1)p_2(w_2)
 \ge\frac{\sqrt{r_1(r_1-1)r_2(r_2-1)}}
 {2\sqrt{(h_1-r_1-1)(h_2-r_2-1)}}
 \ge\frac{\beta(\beta-1)}{n-2\beta-2}.
\]
All denominators are positive. Apply Lemmas~\ref{lem:root-transfer}
and \ref{lem:center-compression}, and use
\eqref{eq:center-distance-identity}. In the vertex case, the extra factor
after including its center vertex in the exponent is
$2/(Rq)=2/(1+q^2)>1$. This proves \eqref{eq:center-universal}.
Finally $\beta\le r_i<R$ implies
$R^2-r_i^2\le(R-\beta)(R+\beta)\le\eta$, and
$N_i>r_i^2-1\ge R^2-\eta-1$.
\end{proof}

\subsection{A finite two-channel estimate}

The first inequality in \eqref{eq:center-compression} loses a factor close
to two when the two components are weakly coupled. The following finite
estimate recovers the needed factor while controlling the other modes.
Only its trace version is needed for the main theorem.

\begin{lemma}[Two-channel resolvent estimate]
\label{lem:center-feshbach}
In the edge-center alternative, let
$\mu_i=\lambda_2(H_i)$ and define the regular root resolvents
\[
 \mathsf h_i(u)=\sum_{j\ge2}
 \frac{p_{i,j}(a_i)^2}{u-\lambda_j(H_i)},\qquad
 \mathsf c_i(u)=-\mathsf h_i'(u)=\sum_{j\ge2}
 \frac{p_{i,j}(a_i)^2}{(u-\lambda_j(H_i))^2}.
\]
Here $(p_{i,j})$ is an orthonormal eigenbasis and $p_{i,1}=p_i$.
Assume
\[
 \beta>\max\{\mu_1,\mu_2\},\qquad
 \mathsf h_1(\beta)\mathsf h_2(\beta)<1.
\]
Write $H_i^*=\mathsf h_i(\beta)$, $C_i^*=\mathsf c_i(\beta)$, and set
\begin{equation}
\label{eq:center-epsilon}
 \varepsilon=
 \frac{C_1^*(\alpha_1^2(H_2^*)^2+\alpha_2^2)
       +C_2^*(\alpha_1^2+\alpha_2^2(H_1^*)^2)}
 {(1-H_1^*H_2^*)^2}.
\end{equation}
Then
\begin{equation}
\label{eq:center-factor-two}
 g\ge\frac{4\alpha_1\alpha_2}{2+\varepsilon}.
\end{equation}
\end{lemma}

\begin{proof}
Let $P$ project onto the two Perron vectors and put $Q=I-P$.
With $v_i=(I-p_ip_i^{\mathsf T})e_{a_i}$, the complementary compression is
\[
 C=QA_TQ=
 \begin{pmatrix}
 A_{H_1}|_{p_1^\perp}&v_1v_2^{\mathsf T}\\
 v_2v_1^{\mathsf T}&A_{H_2}|_{p_2^\perp}
 \end{pmatrix}.
\]
The diagonal blocks of $\beta I-C$ are positive definite.
A rank-one Schur complement shows that $\beta I-C$ is positive definite
precisely when $H_1^*H_2^*<1$, since
$v_i^{\mathsf T}(\beta I-A_{H_i}|_{p_i^\perp})^{-1}v_i=H_i^*$.
Thus every eigenvalue of $C$ is below $\beta$.

Put $D(u)=1-\mathsf h_1(u)\mathsf h_2(u)$.
Eliminating the complement gives the two-dimensional Schur matrix
\[
 S(u)=
 \begin{pmatrix}
 u-r_1-\alpha_1^2\mathsf h_2(u)/D(u)&-\alpha_1\alpha_2/D(u)\\
 -\alpha_1\alpha_2/D(u)&u-r_2-\alpha_2^2\mathsf h_1(u)/D(u)
 \end{pmatrix}.
\]
For example, if $t_i$ are the two Perron-mode coefficients and $\xi_i$ the
complementary root coordinates, their equations are
$\xi_1=\mathsf h_1(u)(\alpha_2t_2+\xi_2)$ and
$\xi_2=\mathsf h_2(u)(\alpha_1t_1+\xi_1)$; solving these gives the formula.
The fixed-space Schur representation yields
\[
 S'(u)=I+B(uI-C)^{-2}B^{\mathsf T}
 \preceq S'(\beta)=I+K\qquad (u\ge\beta),
\]
where $B=PA_TQ$ in orthonormal coordinates. The matrix order follows from
functional calculus for the fixed symmetric matrix $C$. Direct
differentiation of $S$ gives
\[
 K=\frac1{(1-H_1^*H_2^*)^2}
 \left[
 C_1^*\binom{\alpha_1H_2^*}{\alpha_2}
       \begin{pmatrix}\alpha_1H_2^*&\alpha_2\end{pmatrix}
 +C_2^*\binom{\alpha_1}{\alpha_2H_1^*}
       \begin{pmatrix}\alpha_1&\alpha_2H_1^*\end{pmatrix}
 \right].
\]
Thus $K\succeq0$ and $\operatorname{tr}K=\varepsilon$.

Since $\beta$ is the second eigenvalue of $A_T$, the matrix
$\beta I-A_T$ has one negative eigenvalue and at least one zero
eigenvalue. As $\beta I-C\succ0$, Schur inertia shows that
$S(\beta)$ is negative semidefinite. Likewise, $RI-A_T\succeq0$
implies $S(R)\succeq0$. The regular resolvents are nonnegative
and nonincreasing on $[\beta,\infty)$, so $0<D(u)\le1$ there.
For a semidefinite $2$-by-$2$ matrix, the absolute value of its trace is
at least twice the absolute value of its off-diagonal entry. Hence
\[
 \operatorname{tr}\bigl(S(R)-S(\beta)\bigr)
 \ge2\alpha_1\alpha_2
       \left(\frac1{D(R)}+\frac1{D(\beta)}\right)
 \ge4\alpha_1\alpha_2.
\]
On the other hand, integration of $S'(u)\preceq I+K$ gives
\[
 \operatorname{tr}\bigl(S(R)-S(\beta)\bigr)
 \le g\,\operatorname{tr}(I+K)=g(2+\varepsilon).
\]
Combining the two inequalities proves the claim.
\end{proof}

\subsection{The criterion above four}

\begin{theorem}
\label{thm:center-rigidity}
Suppose $T$ minimizes the adjacency gap among all trees of its order.
If $\beta>4$ and the defect in \eqref{eq:center-total-defect} satisfies
$\mathcal E\le3$, then $T$ is a balanced double comet.
\end{theorem}

\begin{proof}
We first lock the squared radii to a common integer degree.
Each selected component has at least $18$ vertices, because
$r_i^2>16$ and $r_i^2\le h_i-1$. Hence $n\ge36$.
Corollary~\ref{cor:small-gap} and $R^2\le n-1$ give
\begin{equation}
\label{eq:center-small-eta}
 g\le n^{-3},\qquad
 \eta=2Rg\le2n^{-5/2}<\frac1{100}.
\end{equation}

For one component temporarily omit subscripts and put
$d=\deg_H(w)$, $s=|N_2(w)|$, and $j=N-d$.
The Perron equation squared at a maximum gives $d\le r^2\le d+s$.
Consequently $j$ is a nonnegative integer with $j\le s$.
Lemma~\ref{lem:distance-two}, with $r>4$ and $E\le3$, gives $s\le9$.

Let $B=A(H-w)$. If $D\ge1$, the number of edges of $H-w$ outside its
remaining root path is exactly
\[
 h-D-d=N+E-d=E+j.
\]
If $D=0$, the residual forest has $E+j-1$ edges instead. Moreover,
$s\le E+j+1$, because at most one distance-two vertex lies on the root path,
and every other one uses a distinct edge outside it.
Lemma~\ref{lem:path-budget}, applied to the component meeting that path,
therefore gives
\begin{equation}
\label{eq:center-residual-prelock}
 \|B\|^2<\Phi(E+j),\qquad
 s\le\min(9,E+j+1),\qquad \Phi(u)=2+\sqrt{u^2+4}.
\end{equation}
For a residual component disjoint from the path, its squared radius is at
most its number of edges. This also proves the displayed norm bound when
$D\le1$, where the remaining path is empty or a singleton.

Suppose $j\ge1$. We have $N\ge16$, $j\le9$, and
$\Phi(E+j)\le\Phi(j+3)<j+6$. Lemma~\ref{lem:rooted-schur} now gives
\[
 j\le r^2-d
 \le\frac{s}{r^2-\|B\|^2}
 <\frac{\min(9,j+4)}{10-j}.
\]
For $1\le j\le4$, $j(10-j)>j+4$; for $5\le j\le9$,
$j(10-j)\ge9$. Both contradict the last strict inequality.
Thus $N_i=\deg_{H_i}(w_i)$ for each component.

The exact residual edge budget consequently improves to $E_i$, or to
$E_i-1$ if $D_i=0$. It follows that
\begin{equation}
\label{eq:center-residual-locked}
 \begin{aligned}
 &\|A(H_i-w_i)\|^2<\Phi(E_i)\le\Phi(3)<6,\qquad
 |N_2(w_i)|\le E_i+1\le4,\\
 &0\le r_i^2-N_i<\frac4{N_i-6}\le\frac25.
 \end{aligned}
\end{equation}
Intervals $[N_i,N_i+2/5)$ for different integer $N_i$ are separated by at
least $3/5$. Equations~\eqref{eq:center-integer-budget} and
\eqref{eq:center-small-eta} therefore force
$N_1=N_2=N\ge16$. Since $h_i\ge N+1$, we have $n\ge2N+2$ and hence
$\eta\le1/N^2$.

The mass estimate in Lemma~\ref{lem:rooted-schur} gives
\begin{equation}
\label{eq:center-mass-locked}
 p_i(w_i)^2\ge
 \left(2+\frac8{(N-6)^2}\right)^{-1}
 \ge\frac{25}{52}>\frac{12}{25}.
\end{equation}
At least one of $E_1,E_2$ is at most one. For that component the residual
squared norm is below $\Phi(1)<5$ and $|N_2(w_i)|\le2$. Therefore
\begin{equation}
\label{eq:center-radius-window}
 \begin{aligned}
 0\le R^2-N&\le\frac2{N-5}+\frac1{N^2}
 \le B_N:=\frac9{4(N-5)},\\
 N-3&<y<N-2+B_N<N-1.
 \end{aligned}
\end{equation}

We next control the other spectral modes in the edge-center alternative.
Write $s_i=\alpha_i^2\le1/2$. Bipartite symmetry gives equal root spectral
weights at the two Perron eigenvalues $\pm r_i$.
Interlacing in \eqref{eq:center-residual-locked} places every other
eigenvalue in $(-\sqrt6,\sqrt6)$. Pairing the corresponding root spectral
weights, and retaining the negative Perron pole separately, gives at
$u=\beta>4$
\begin{equation}
\label{eq:center-regular-four}
 \begin{split}
 H_i^*&\le(1-2s_i)\frac25+\frac{s_i}{8}\le\frac25,\\
 C_i^*&\le(1-2s_i)\frac{11}{50}+\frac{s_i}{64}
       \le\frac{11}{50}-\frac25s_i.
 \end{split}
\end{equation}
Indeed the kernels per unit paired spectral mass are
$u/(u^2-\theta^2)$ and $(u^2+\theta^2)/(u^2-\theta^2)^2$.
They increase with $\theta^2$ and decrease with $u>\sqrt6$, so evaluation
at $\theta^2=6$, $u=4$ gives these bounds.
In particular $H_1^*H_2^*\le4/25<1$, and
Lemma~\ref{lem:center-feshbach} applies.

The numerator in \eqref{eq:center-epsilon} is at most
\[
\begin{aligned}
 &\left(\frac{11}{50}-\frac25s_1\right)
       \left(\frac4{25}s_1+s_2\right)
 +\left(\frac{11}{50}-\frac25s_2\right)
       \left(s_1+\frac4{25}s_2\right)\\
 &\quad=\frac{29}{25}
       \left[\frac{11}{50}(s_1+s_2)-\frac45s_1s_2\right]
       -\frac8{125}(s_1-s_2)^2
 \le\frac{319}{2500}.
\end{aligned}
\]
Here the expression in square brackets is bilinear in $(s_1,s_2)$,
so its maximum on $[0,1/2]^2$ occurs at a vertex and equals $11/100$.
Since $(1-H_1^*H_2^*)^2\ge(21/25)^2$, the trace bound gives
\[
 \varepsilon\le\frac{319}{1764}<\frac15,
 \qquad
 g\ge\frac{20}{11}\alpha_1\alpha_2
   >\frac{12}{7}\alpha_1\alpha_2.
\]
For a vertex center use Lemma~\ref{lem:center-compression} instead.
After its extra vertex is included in the path exponent, its multiplier is
$2/(Rq)=2/(1+1/y)>13/7>12/7$ by \eqref{eq:center-radius-window}.

Set $A=n-2N\ge2$. Applying Lemma~\ref{lem:root-transfer} and
\eqref{eq:center-mass-locked} in either alternative gives
\begin{equation}
\label{eq:center-gap-four}
 g\ge\frac{12}{7}\frac{12}{25}
       \left(\frac{12}{13}\right)^2y^{-(A-\mathcal E)/2}
 >c\,y^{-(A-\mathcal E)/2},\qquad c=\frac7{10}.
\end{equation}
The coefficient comparison is $207360>207025$.

It remains to exclude a positive defect by comparing adjacent comet
parameters. The same-order balanced comets
$C_N=C(N-1,A+2)$ and $C_{N+1}=C(N,A)$ are feasible. By
Theorem~\ref{thm:comet-upper},
\begin{equation}
\label{eq:center-adjacent-comets}
 g(C_N)\le(N-2)^{-A/2},\qquad
 g(C_{N+1})\le(N-1)^{-(A-2)/2}.
\end{equation}
Suppose $\mathcal E\ge1$. If $y\le N-2$, the first comparator and
\eqref{eq:center-gap-four} give
$g/g(C_N)>c\sqrt y>c\sqrt{13}>1$, a contradiction.
Thus $y>N-2$. Minimality and the second comparator imply
\[
 \frac A2\log\frac{N-1}{y}
 \le\log\frac{N-1}{c}-\frac{\mathcal E}{2}\log y
 <\log\frac{N-1}{c}.
\]
Using \eqref{eq:center-radius-window} and the elementary bounds
$\log(b/a)\le(b-a)/a$ and $\log(b/a)\ge(b-a)/b$ for $b\ge a>0$,
we obtain
\begin{equation}
\label{eq:center-transfer-loss}
 \frac A2\log\frac y{N-2}
 \le\frac{9(N-1)}{(N-2)(4N-29)}
       \log\frac{N-1}{c}<\frac67.
\end{equation}
For the last strict inequality, the expression decreases for real $N\ge16$:
both $(N-1)/(N-2)$ and
$\log((N-1)/c)/(N-29/4)$ decrease. The derivative of the latter has
numerator
$(N-29/4)/(N-1)-\log((N-1)/c)<1-\log(150/7)<0$.
At $N=16$ the expression is
$(27/98)\log(150/7)<(27/98)(28/9)=6/7$,
since $e^{28/9}>20(10/9)>150/7$.

Finally, \eqref{eq:center-gap-four}--\eqref{eq:center-transfer-loss} give
\[
 \frac{g}{g(C_N)}>c\sqrt y\,e^{-6/7}
 >\frac7{10}\frac{18}{5}\frac5{12}=\frac{21}{20}>1.
\]
Here $\sqrt{13}>18/5$ and $e^{6/7}<12/5$.
For explicit certificates, the exponential series gives
$e^3>\sum_{j=0}^{8}3^j/j!>20$ and
$e<1+1+1/2+(1/6)/(1-1/4)=49/18<11/4$;
then $11^6 5^7<4^6 12^7$ proves $e^{6/7}<12/5$.
These estimates also justify the preceding logarithm bound.

The contradiction forces $\mathcal E=0$. Thus $E_1=E_2=z=0$.
We already know that $\deg_{H_i}(w_i)=N$, so $h_i=N+D_i$.
Since $N\le r_i^2\le h_i-1$, we have $D_i\ge1$.
The $a_i$--$w_i$ path has $D_i$ edges, and exactly one of them is
incident to $w_i$. Together with the other $N-1$ edges incident to $w_i$,
these account for all $D_i+N-1=h_i-1$ edges of $H_i$.
Thus every vertex outside the path is a leaf adjacent to $w_i$.
Both selected components are rooted brooms with common hub degree $N$
and positive handle lengths. Joining the handles across the edge or
through the center vertex gives exactly the balanced double comet
$C(N-1,n-2N+2)$.
\end{proof}

\section{The global comparison}\label{sec:global}

We now prove that every gap minimizer of order at least $200$ satisfies
the hypotheses of Theorem~\ref{thm:center-rigidity}. Throughout this
section, $T$ is a gap minimizer of order $n\ge200$,
\[
 R=\lambda_1(T),\qquad \beta=\lambda_2(T),\qquad
 g=R-\beta,\qquad \ell=\log n,\qquad
 \eta=2Rg,\qquad \eta_n=2n^{-5/2}.
\]
All logarithms are natural. Whenever $R>2$, put
\[
 q=\frac{R-\sqrt{R^2-4}}2,\qquad y=q^{-2};
 \qquad R^2=y+2+y^{-1}.
\]
The spectral-center quantities $N_i,E_i,\mathcal E$ are as in
Section~\ref{sec:center}. In particular,
\begin{equation}\label{eq:global-floor}
 N_1+N_2\ge 2R^2-2\eta-2.
\end{equation}
Indeed $r_i\ge\beta$, $R^2-\beta^2\le\eta$, and
$\lfloor r_i^2\rfloor\ge r_i^2-1$.
Lemma~\ref{lem:center-bound} gives
\begin{equation}\label{eq:global-input}
 g\ge
 \frac{\beta(\beta-1)}{n-2\beta-2}
 (1-y^{-1})^2
 y^{-(n-N_1-N_2-\mathcal E)/2}.
\end{equation}
Lemma~\ref{lem:isolation} and Corollary~\ref{cor:small-gap} give
\begin{equation}\label{eq:global-isolation}
 \beta>2,\qquad R>2+\frac1{162},\qquad
 g\le n^{-3},\qquad 0<\eta\le\eta_n.
\end{equation}
The constants used below are verified in Appendix~\ref{app:constants}.

\subsection{A fixed comparison tree}

\begin{lemma}\label{lem:bootstrap-four}
If $n\ge200$, then $\lambda_2(T)>4$.
\end{lemma}

\begin{proof}
The function $u+2+u^{-1}$ increases for $u>1$.
Since its value at $8/7$ is $4+1/56<4+2/81$,
\eqref{eq:global-isolation} gives $y>8/7$.
Put
\[
 a=n/2-1+\eta,\qquad
 \Phi_a(u)=(a-u)\log u-2\log(1-u^{-1})\quad(u>1).
\]
Since $\beta>2$, one has
$\beta(\beta-1)/(n-2\beta-2)>2/n$.
Combining \eqref{eq:global-floor} and \eqref{eq:global-input},
using $\mathcal E\ge0$ and discarding the nonpositive term
$-y^{-1}\log y$, gives
\begin{equation}\label{eq:bootstrap-coarse}
 \log(1/g)\le\Phi_a(y)+\log(n/2).
\end{equation}

Suppose that $\beta\le4$. Then $R^2\le16+\eta<17$, so $y<15$ and
\[
 y=R^2-2-y^{-1}<14+\eta-1/15<14,
\]
because $\eta\le\eta_n<1/15$.
For $8/7\le u\le14$, the bound $\log14<8/3$ yields
\[
 u\Phi_a'(u)
 =a-u(1+\log u)-\frac2{u-1}
 >99-14\left(1+\frac83\right)-14
 =\frac{101}{3}>0.
\]
Thus $\Phi_a$ increases on this interval. Since $(14/13)^2<2$,
\eqref{eq:bootstrap-coarse} implies
\begin{equation}\label{eq:bootstrap-refined}
 \log(1/g)\le(n/2-15+\eta)\log14+\log n.
\end{equation}
The feasible same-order comparator $C(25,n-50)$ gives
\[
 \log(1/g)\ge(n/2-26)\log24
\]
by Theorem~\ref{thm:comet-upper}.
The difference between this lower bound and the upper bound in
\eqref{eq:bootstrap-refined} is at least
\[
 D_4(n)=(n/2-26)\log24-(n/2-15)\log14
        -\log n-\eta_n\log14.
\]
Since $\eta_n$ decreases,
\[
 D_4'(n)>\frac12\log(12/7)-\frac1n
 >\frac5{24}-\frac1n>0\qquad(n\ge200).
\]
At $n=200$, the estimates in Appendix~\ref{app:constants} and
$\eta_{200}\log14<1/1000$ give
\[
 D_4(200)>
 74\frac{19}{6}-85\frac83-\frac{16}{3}-\frac1{1000}
 =\frac73-\frac1{1000}>0.
\]
This contradicts \eqref{eq:bootstrap-refined} for every $n\ge200$.
\end{proof}

\subsection{The optimized comparison and the integer defect}

For $a>1$ and $u>0$, define
\[
 f_a(u)=(a-u)\log u,\qquad F(a)=\max_{u>0}f_a(u).
\]
Let $H(u)=u(1+\log u)$ for $u>1$.
Since $H$ is strictly increasing from $1$ to infinity, there is a
unique $\nu_a>1$ satisfying
\[
 H(\nu_a)=\nu_a(1+\log\nu_a)=a.
\]
The function $f_a$ is strictly concave, and
$f_a'(\nu_a)=0$, so $\nu_a$ is its unique maximizer.
Consequently, $\nu_a$ increases with $a$ and $F'(a)=\log\nu_a$.

The equivalent formula $\nu_a=a/W(ea)$ uses the positive real branch
of the Lambert function.

\begin{lemma}\label{lem:rounded-comet}
For every $n\ge200$, putting $b=n/2-2$, one has
\begin{equation}\label{eq:rounded-comet}
 \log(1/g)\ge F(b)-\epsilon_n,
 \qquad \epsilon_n=\frac{3(\log n)^2}{n}.
\end{equation}
\end{lemma}

\begin{proof}
Round $\nu_b$ to a nearest integer $j$.
Since $b\ge2e=H(e)$, one has $\nu_b\ge e$ and hence
$\nu_b=b/(1+\log\nu_b)\le b/2<n/e$.
It follows that $1+\log\nu_b<\ell$, so
\[
 \frac{n}{3\ell}\le\frac b\ell<\nu_b\le b/2.
\]
In particular $j\ge2$ and $j\le b$, so
$C(j+1,n-2j-2)$ is feasible. On the half-unit rounding interval,
$u\ge\nu_b/2\ge n/(6\ell)$, and hence
\[
 |f_b''(u)|=b/u^2+1/u\le24\ell^2/n.
\]
Taylor's theorem gives $f_b(j)\ge F(b)-\epsilon_n$.
Theorem~\ref{thm:comet-upper} gives
$g\le\exp(-f_b(j))$, proving \eqref{eq:rounded-comet}.
\end{proof}

\begin{lemma}\label{lem:global-defect}
For every $n\ge200$, the total spectral-center defect of a gap
minimizer satisfies $\mathcal E\le3$.
\end{lemma}

\begin{proof}
Put $b=n/2-2$. By Lemma~\ref{lem:bootstrap-four}, $\beta>4$ and $y>13$.
Since $y\le\beta^2+\eta-2$,
\[
 6\beta(\beta-1)-5y
 \ge(\beta-3)^2+1-5\eta>0.
\]
Thus $\beta(\beta-1)>(5/6)y$.
If $\mathcal E\ge4$, equations
\eqref{eq:global-floor}--\eqref{eq:global-input} imply
\[
 g\ge\frac5{6n}(1-y^{-1})^2
       \exp\bigl(-f_{b-2+\eta}(y)\bigr)
 \ge\frac5{6n}(1-y^{-1})^2
       \exp\bigl(-F(b-2+\eta)\bigr).
\]
We have dropped the additional factor $y^{1/y}\ge1$.
Comparison with \eqref{eq:rounded-comet}, using
$-2\log(1-y^{-1})\le2/(y-1)<1/6$, forces
\begin{equation}\label{eq:necessary-defect}
 F(b)-F(b-2+\eta)
 <\log(6n/5)+3\ell^2/n+1/6.
\end{equation}

Since $\ell>5$ and $e^2<8$,
\[
 H(n/e^2)=\frac n{e^2}(\ell-1)>n/2>b-2.
\]
Thus $\nu_{b-2}<n/e^2$, and the defining equation for $\nu_{b-2}$
gives
\[
 \nu_{b-2}=\frac{b-2}{1+\log\nu_{b-2}}
 >s(n),\qquad s(n)=\frac{n/2-4}{\ell-1}.
\]
For $n\ge200$, one has $1<s(n)<n$ and $\log n\le\sqrt n$, so
\[
 \eta_n\log s(n)\le\frac2{n^2}<\frac1{1000}.
\]
By monotonicity of $F'$,
\begin{equation}\label{eq:F-difference}
 F(b)-F(b-2+\eta)
 \ge(2-\eta_n)\log s(n)
 >2\log s(n)-\frac1{1000}.
\end{equation}
It therefore suffices to prove that
\[
 K(n)=2\log s(n)-\log(6n/5)-3\ell^2/n-1/6
 >\frac1{1000}.
\]
The function $K$ strictly increases for real $n\ge200$, because
\[
 \frac{d}{dn}\bigl(2\log s(n)-\log(6n/5)\bigr)
 =\frac2{n-8}-\frac2{n(\ell-1)}-\frac1n>0
\]
and $3\ell^2/n$ decreases.
At $n=200$, Appendix~\ref{app:constants} gives
\[
 s(200)>\frac{288}{13}>e^{37/12},
 \qquad \log200<\frac{16}{3}.
\]
Using also $\log(6/5)<1/5$, we obtain
\[
 K(200)>
 \frac{37}{6}
 -\left(\frac{16}{3}+\frac15+\frac{32}{75}+\frac16\right)
 =\frac1{25}>\frac1{1000}.
\]
Consequently, \eqref{eq:F-difference} contradicts
\eqref{eq:necessary-defect} for every $n\ge200$.
\end{proof}

\begin{proof}[Proof of Theorem~\ref{thm:main}]
By Lemmas~\ref{lem:bootstrap-four} and~\ref{lem:global-defect},
every gap minimizer of order $n\ge200$ has $\beta>4$ and
$\mathcal E\le3$. Theorem~\ref{thm:center-rigidity} therefore
identifies it as a balanced double comet.
\end{proof}

\section{Further questions}
\label{sec:further}

The finite threshold comes from two separate steps: forcing
$\lambda_2>4$ and then forcing the total center defect to be at most
three. The center criterion itself applies at any order satisfying
these hypotheses. Thus one way to lower the threshold is to strengthen
the global comparison, while another is to prove an analogous rigidity
criterion at smaller spectral radii. The orders below $200$ are not
resolved by the present proof.

Even after the extremal family is known, selecting its optimal integer
parameter remains a separate problem. The continuous maximizer defined by
\[
 \nu_b(1+\log\nu_b)=b,\qquad b=n/2-2,
\]
and used in Section~\ref{sec:global} selects an effective comparison tree.
Its rounding does not by itself determine the exact minimizer of
$g(C(k,n-2k))$. Establishing that integer choice and any uniqueness
statement would strengthen Theorem~\ref{thm:main}.

Theorem~\ref{thm:two-boundary} applies to every finite irreducible
reversible continuous-time chain. Its two-sided estimate becomes accurate when
the Dirichlet gap is large compared with the global gap.
The Schur derivative in Proposition~\ref{prop:two-vertex} identifies
the relevant variance without computing a second eigenvector.
It would be useful to determine which other extremal graph problems
admit a structural argument combining this representation with an
integer budget of the kind used here.

\appendix
\section{Elementary numerical estimates}\label{app:constants}

For clarity, we record rational certificates for the logarithmic estimates
used in the finite threshold argument. The exponential series gives
\[
 \frac{27}{10}<e<\frac{65}{24}+\frac1{100}
 =\frac{1631}{600}<\frac{68}{25}<\frac{11}{4}.
\]
For the upper bound, the tail starting at $1/5!$ is at most
$(1/120)/(1-1/6)=1/100$. Also
\[
 e^3>\sum_{j=0}^8\frac{3^j}{j!}
 =\frac{89641}{4480}>20,
 \qquad e^3<\left(\frac{68}{25}\right)^3<\frac{81}{4}.
\]
The series
\[
 \log\frac{1+t}{1-t}
 =2\sum_{j=0}^{\infty}\frac{t^{2j+1}}{2j+1}
 \qquad(0<t<1)
\]
at $t=1/3$, truncated after four terms, gives
$\log2>693/1000$; at $t=1/2$, three terms give
$\log3>109/100$. Consequently
\[
 \log24=3\log2+\log3>\frac{3169}{1000}>\frac{19}{6}.
\]
Similarly, the positive exponential series yields
\[
 e^2>\frac{22}{3},\qquad e^{2/3}>\frac{157}{81},
 \qquad e^{8/3}>\frac{3454}{243}>14,
\]
so $\log14<8/3$. The inequalities $e^4>54>48$ give
$\log48<4$.

The bounds on $\log200$ follow from
\[
 e^5<\left(\frac{11}{4}\right)^5<200,
 \qquad
 e^{16/3}=e^3e^2e^{1/3}
 >20\cdot\frac{22}{3}\cdot\frac{25}{18}>200.
\]
Finally, $e^{1/12}<\sum_{j\ge0}(1/12)^j=12/11$ implies
\[
 e^{37/12}<\frac{81}{4}\frac{12}{11}
 =\frac{243}{11}<\frac{288}{13}.
\]
All estimates are strict. The remaining bounds
$\log(1+t)<t$ and $\log(1+t)>t/(1+t)$, for $t>0$,
follow directly by integrating $1/(1+t)$.


\begin{thebibliography}{9}

\bibitem{BurdaDudaLuckWaclaw2009}
Z.~Burda, J.~Duda, J.~M.~Luck, and B.~Waclaw,
\emph{Localization of the maximal entropy random walk},
Phys. Rev. Lett. \textbf{102} (2009), no.~16, 160602.
\href{https://doi.org/10.1103/PhysRevLett.102.160602}
{doi:10.1103/PhysRevLett.102.160602}.

\bibitem{DorflerBullo2013}
F.~D\"orfler and F.~Bullo,
\emph{Kron reduction of graphs with applications to electrical networks},
IEEE Trans. Circuits Syst. I Regul. Pap. \textbf{60} (2013), no.~1,
150--163.
\href{https://doi.org/10.1109/TCSI.2012.2215780}
{doi:10.1109/TCSI.2012.2215780}.

\bibitem{Hou2026}
Wenxin Hou,
\emph{Minimum spectral gap of trees and graphs},
contributed talk abstract, 34th Cumberland Conference on Combinatorics,
Graph Theory, and Computing, Auburn University, May 16--17, 2026,
conference program, p.~10.

\bibitem{JovovicKoledinStanic2018}
I.~Jovovi\'c, T.~Koledin, and Z.~Stani\'c,
\emph{Trees with small spectral gap},
Ars Math. Contemp. \textbf{14} (2018), no.~1, 197--207.
\href{https://doi.org/10.26493/1855-3974.992.68d}
{doi:10.26493/1855-3974.992.68d}.

\bibitem{KumarMoharPragadaZhan2025}
H.~Kumar, B.~Mohar, S.~Pragada, and H.~Zhan,
\emph{On the second largest adjacency eigenvalue of trees with given
diameter},
Linear Multilinear Algebra \textbf{73} (2025), no.~13, 2944--2972.
\href{https://doi.org/10.1080/03081087.2025.2484265}
{doi:10.1080/03081087.2025.2484265}.

\bibitem{KumarMoharPragadaZhan2026}
H.~Kumar, B.~Mohar, S.~Pragada, and H.~Zhan,
\emph{Convex combination of first and second eigenvalues of trees},
preprint, 2026,
\href{https://arxiv.org/abs/2601.10036v2}{arXiv:2601.10036v2}.

\end{thebibliography}
\end{document}